\documentclass[12pt]{article}
\usepackage{amsmath, verbatim}
\usepackage{amssymb}
\usepackage{amsthm}
\usepackage{amscd}
\usepackage{amsfonts}
\usepackage{graphicx}
\usepackage{fancyhdr}
\usepackage{epsfig}
\usepackage{makeidx}
\usepackage{setspace}
\usepackage{tikz}
\usepackage{pstricks}%,pst-fractal,pstricks-add}

\theoremstyle{plain} \numberwithin{equation}{section}
\newtheorem{theorem}{Theorem}[section]
\newtheorem{corollary}[theorem]{Corollary}

\newtheorem{proposition}[theorem]{Proposition}

\theoremstyle{definition}
\newtheorem{definition}[theorem]{Definition}

\newtheorem{remark}[theorem]{Remark}
\newtheorem{example}[theorem]{Example}
\newtheorem{question}{Question} %\topmargin-2cm
\newtheorem*{acknowledgements}{Acknowledgements}

\newcommand{\mbf}[1]{\mathbf{#1}}

\newcommand{\ds}{{\displaystyle}}

\newcommand{\mdeg}{{\textup{mdeg}}}

\DeclareMathOperator{\LCM}{\mathbf{LCM}}

\DeclareMathOperator{\lcm}{lcm}

\title {Rigid monomial ideals}
\author {Timothy B.P. Clark and Sonja Mapes}

\begin{document}

\maketitle 
\begin{abstract}
In this paper we investigate the class of rigid monomial ideals.  We give a characterization 
of the minimal free resolutions of certain classes of these ideals.  
Specifically, we show that the ideals in a particular subclass of rigid monomial ideals 
are lattice-linear and thus their minimal resolution can be 
constructed as a poset resolution.  We then use this result to give 
a description of the minimal free resolution of a larger class of rigid monomial ideals 
by using $\mathcal{L}(n)$, the lattice of all lcm-lattices of monomial ideals 
with $n$ generators.  By fixing a stratum in $\mathcal{L}(n)$ 
where all ideals have the same total Betti numbers we show that 
rigidity is a property which is upward closed in $\mathcal{L}(n)$.  
Furthermore, the minimal resolution of all rigid ideals contained 
in a fixed stratum is shown to be isomorphic to the constructed minimal resolution.
\end{abstract}

\section*{Introduction} 
  
Giving a constructive method for finding the minimal free resolution of a monomial ideal 
is a question which has motivated a wide variety of projects in commutative algebra.  
Various methods are known for computing the multigraded Betti numbers in the general case, 
and there are numerous strategies for constructing the maps in free resolutions.  
Despite this diversity of results it is still not known how to construct the maps 
in a minimal free resolution except for certain subclasses of monomial ideals (generic and Borel are examples).    

In this paper we explore the class of \emph{rigid} monomial ideals which were introduced to us by Ezra Miller \cite{MilPev}.  By definition, a rigid ideal has 
the following two properties; (R1) every nonzero multigraded Betti number equals 1 and (R2) multigraded Betti numbers which are nonzero in the same homological 
degree correspond to incomparable monomials in the lcm-lattice.  
Rigid ideals include as subclasses the generic monomial ideals 
and the monomial ideals whose minimal resolution is supported on their Scarf complex.  
Furthermore, rigid ideals are a generalization of the monomial ideals in three variables whose minimal resolution may 
be constructed using a rigid embedding, introduced by Miller \cite{Miller}.

Theorem \ref{uniqueRes}, which was communicated along with the definition \cite{MilPev}, characterizes rigid monomial ideals as having a unique finely graded minimal resolution 
up to independently rescaling the finely graded basis vectors.  As such, the maps in such a minimal resolution should be explicitly describable.  Our interest in studying rigid ideals stems from a desire to identify combinatorial objects which encode the data of these maps.  We propose that the lcm-lattice or more specifically the subposet of the lcm-lattice consisting of Betti degrees is the appropriate combinatorial object.  The techniques of this paper allow us to prove such a statement for a subclass of rigid monomial ideals. 
We give a detailed description of the obstruction to using these new methods on all rigid monomial ideals. 

More precisely, given that the minimal resolution of a rigid monomial ideal is unique up to scaling, and that the Betti numbers can be computed from the lcm-lattice it seems that a description of the resolution's maps using the relations in the lcm-lattice should be possible.  We aim toward this goal by taking advantage of the construction described by the first author in \cite{Clark}.  This construction takes as its input the lcm-lattice of a monomial ideal and produces an approximation to the minimal free resolution of the given ideal.  In the case when the minimal free resolution is indeed obtained, the ideal is said to be \emph{lattice-linear}.  More generally, if a multigraded poset is used as input for this construction and the resulting sequence is an exact complex of multigraded modules, it is called a \emph{poset resolution}.  

In this paper we focus attention on a subclass of rigid ideals which we call \emph{concentrated}.  In general, an ideal is said to be concentrated if its Betti (contributing) multidegrees are less than all the multidegrees in the lcm-lattice which are not Betti multidegrees.  In particular, Theorem \ref{concentratedLatticeLinear} states that a rigid monomial ideal is concentrated if and only if it is lattice-linear.  We therefore construct a minimal poset resolution for the class of concentrated rigid monomial ideals.  

The remainder of the paper focuses on the development of a method for 
transferring the resolution information of a rigid monomial ideal to 
related ideals which have the same total Betti numbers.  In particular, we 
consider rigid ideals in relation to their neighbors in $\mathcal{L}(n)$, 
which is the set of all finite atomic lattices (or lcm-lattices) with $n$ ordered atoms.  
In \cite{phan}, Theorem 4.2 shows that under an appropriate partial order, $\mathcal{L}(n)$ is 
itself a finite atomic lattice.  Furthermore, Theorem 3.3 in \cite{GPW} implies that total Betti numbers 
increase along chains in $\mathcal{L}(n)$.  It is therefore natural to study subposets of $\mathcal{L}(n)$ 
with fixed total Betti numbers, which we refer to individually as Betti stratum.  
Within a fixed Betti stratum it is useful to describe families of monomial ideals within 
these subposets with isomorphic minimal resolutions (while allowing non isomorphic lcm-lattices).  Precisely, 
Theorem \ref{resInRigidFilter} states that any two rigid monomial ideals in a fixed Betti stratum 
whose lcm-lattices are comparable in $\mathcal{L}(n)$ must have isomorphic minimal resolutions.  
This allows us to construct a minimal poset resolution for a larger class of rigid monomial ideals than 
guaranteed by Theorem \ref{concentratedLatticeLinear}.

The final section of this paper discusses precisely how the methods of this paper can fail for 
rigid monomial ideals which are not concentrated.  It should be noted that this failure is theoretical, 
for we have yet to construct examples of rigid monomial ideals whose minimal resolution 
cannot be constructed as a poset resolution on the subposet of Betti degrees in the lcm-lattice.

\begin{acknowledgements}
We would like to thank Ezra Miller and Irena Peeva for comments which
helped improve the clarity of this paper and for their fundamental definitions and observations for without which this paper would not exist.
\end{acknowledgements}

\section{Rigid monomial ideals}\label{IntroRigid}
Let $R=\Bbbk[x_1,\ldots,x_d]$.  We write 
$\mbf{x}^\mbf{a}=x_1^{a_1}\cdots x_d^{a_d}$ for a monomial in 
$R$ and investigate properties of ideals in $R$ which are 
generated by monomials.  

Recall that for a monomial ideal $M$ whose set of minimal generators is $G(M)=\{m_1,\ldots,m_n\}$, 
the finite atomic lcm-lattice $\LCM(M)$ has the monomials of $G(M)$ as its atoms and 
the least common multiples of $m_1,\ldots,m_n$ as its ground set.  The monomials in $\LCM(M)$ are 
ordered by divisibility, the maximal element of $\LCM(M)$ is $\lcm(m_1,\ldots,m_n)$ 
and the minimal element is 1, considered as the lcm of the empty set of monomials.  
Gasharov, Peeva and Welker first defined this combinatorial object in \cite{GPW} and derived 
a formula for the multigraded Betti numbers of $M$ based on the homology of its open intervals.   

The following definition was communicated to the second author by Miller \cite{MilPev}.  

\begin{definition}\label{RigidDef}
A rigid monomial ideal is a monomial ideal whose multigraded 
Betti numbers satisfy the following two properties:
\begin{enumerate}
\item[(R1)] $\beta_{i,\mbf{b}}$ is either 1 or 0 for all $i$ and all multidegrees $\mbf{b}$.
\item[(R2)] If $\beta_{i,\mbf{b}} = 1$ and $\beta_{i,\mbf{b'}} = 1$ 
						then $\mbf{x^b}$ and $\mbf{x^{b'}}$ are not comparable in $\LCM(M)$, the lcm-lattice of $M$.
\end{enumerate}
\end{definition}

Note that since (multigraded) Betti numbers are dependent on the characteristic of the field $\Bbbk$, the condition of being 
rigid is as well.  In fact, the well-known example of the monomial ideal arising from a triangulation 
of the real projective plane is rigid in characteristic other than 2 and not rigid when the characteristic is 2.

\begin{example}
The monomial ideal $M = (a^2, ab, b^2)$ is rigid since the first syzygies $a^2b$ and 
$ab^2$ are not comparable in $\LCM(M)$, whereas the monomial ideal $N= (bc, ac, a^2b)$ is not 
rigid since the first syzygies $abc$ and $a^2bc$ are comparable in $\LCM(N)$.  

\begin{figure}
\center
\includegraphics[scale=0.5]{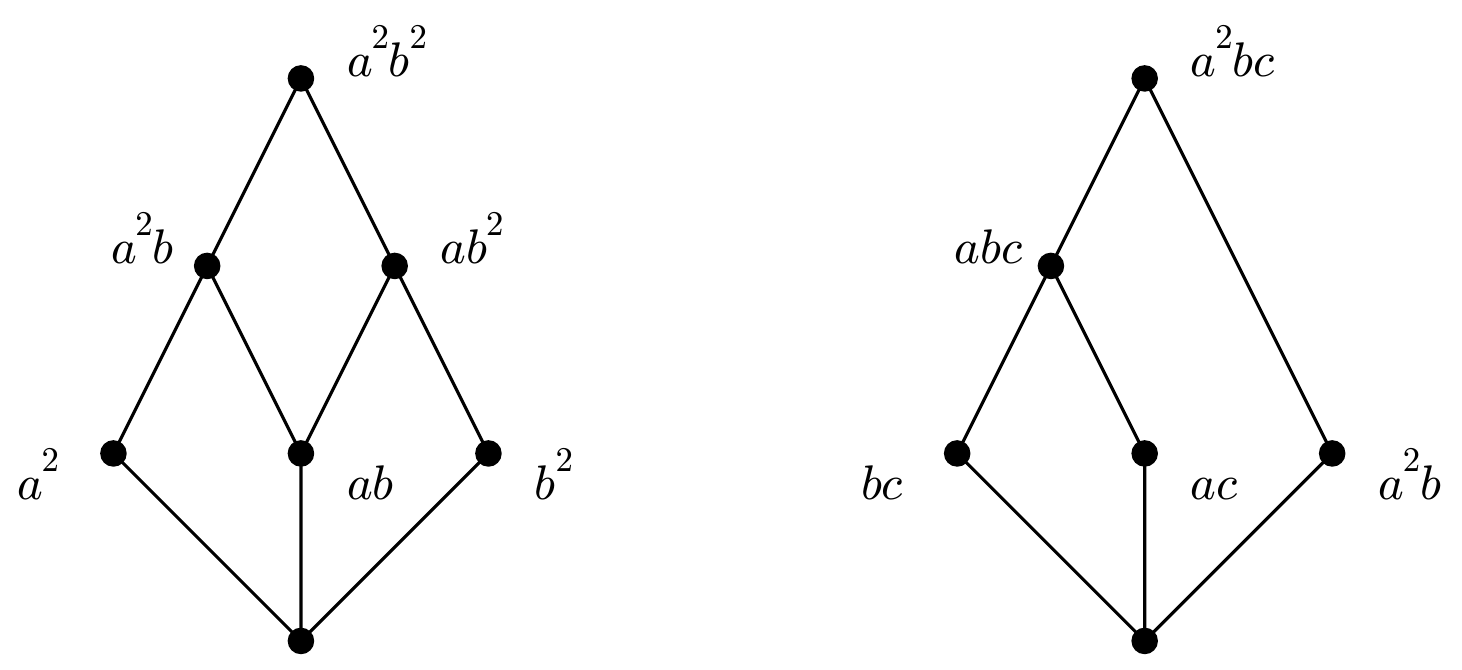}
\caption{lcm-lattices of monomial ideals $M$ (rigid) and $N$ (non-rigid)}\label{rigidExamples}
\end{figure}

\end{example}

Our main interest in studying rigid monomial ideals follows 
from the fact that their minimal resolutions are unique up to scaling.  
To state precisely and prove this property, we first discuss 
the automorphisms of minimal resolutions of a multigraded module.

\begin{definition} 
Let $\mathcal{F}$ be a multigraded 
free resolution of a multigraded module. An \emph{automorphism} 
of $\mathcal{F}$ is a collection of multigraded (degree 0) isomorphisms 
$f_i:F_i\rightarrow F_i$ which has the property that  
$d_i'\circ f_i=f_{i-1}\circ d_i$ for every $i\ge 1$.  Here, 
$\{d_i\}$ and $\{d_i'\}$ are the representatives of the differential of 
$\mathcal{F}$ which come as a result of distinct basis choices.  
\end{definition} 

For an arbitrary monomial ideal, the isomorphism $f_i$ may be realized as 
an element of $GL_{\beta_i}(R)$ for every $0 \le i \le p$  
and as such, the automorphism group of $\mathcal{F}$ 
is a subgroup of $\ds\bigoplus_{i=0}^p GL_{\beta_i}(R)$.  
We say that an automorphism of $\mathcal{F}$ is \emph{trivial} 
if the maps representing the isomorphisms $\{f_i\}$ are 
diagonal matrices with units of $\Bbbk$ along the diagonal.  
The following characterization of rigid monomial ideals 
was also communicated to the second author by Miller \cite{MilPev} without proof.  We restate the result here and provide our own proof since it does not appear in the literature.

\begin{proposition}\label{uniqueRes}
The automorphisms of the minimal resolution of $R/M$ are 
trivial if and only if $M$ is rigid.
\end{proposition}

\begin{proof}

Write $\mathcal{F}$ for the minimal free resolution of $R/M$.  
To consider automorphisms on $\mathcal{F}$ we address on 
a case by case basis the automorphisms of the free module $F_i$ which are possible.  
The fact that the maps $f_i$ must commute with the differentials gives 
the condition $$d_i = f_{i-1}^{-1} \circ d'_i \circ f_{i}$$ for $1\le i\le p$ .  
As such, we consider how each $f_i$ acts on $d'_i$.

The minimal free resolution of an arbitrary monomial ideal has $F_0$, a 
free module of rank one, appearing in homological degree zero of its minimal resolution.  
This module has a basis element of multidegree $\mbf{0}=(0,\ldots,0)$  
and as such, any automorphism $f_0$ is represented by 
a $1\times 1$ matrix whose entry is a unit of $\Bbbk$.  

Since the $r$ minimal generators of a monomial ideal are unique, the free 
module $F_1$ consists of $r$ free modules of rank one whose individual 
shifts match the multidegrees of the monomial generators.  
Any automorphism of the free module $F_1$ therefore may only send 
a basis element of multidegree $\mbf{a}$ to a scalar multiple of itself.  
We therefore see that that there is only one choice for $d_1$ (up to scaling) 
and hence, $f_1$ is a matrix with units along its diagonal.  

Note that we have not as yet assumed rigidity, 
so that for any automorphism of the minimal free resolution 
of a monomial ideal, the individual module automorphisms 
$f_0$ and $f_1$ are trivial.  

In what follows, we assume that $M$ is a rigid monomial ideal and proceed 
with the proof that the automorphisms of $\mathcal{F}$ are trivial by 
induction on homological degree.  

Let $i>1$ and suppose that $f_{i-1}$ is a matrix with units along 
the diagonal, so that we need to show the same is true for $f_i$.  
Since $f_i$ is a multigraded isomorphism, it is possible to more precisely 
describe how its matrix entries relate to the multidegrees of the 
source and target basis elements, whose relationship $f_i$ encodes.  
Indeed, write $\mbf{b}$ for the multidegree of a basis element $b$ in 
the free module $F_i$.  In order for the map $f_i$ to preserve 
the multidegree $\mbf{b}$, the monomial $\mbf{x^b}$ must be divisible by 
the monomials whose multidegrees index the rows of $f_i$ and which 
correspond to the nonzero entries of the column of $f_i$ that is of multidegree $\mbf{b}$.  
Precisely, for a basis element $b\in F_i$ of multidegree $\mbf{b}$, we 
write $b=m_{b,c_1}c_1+\cdots+m_{b,c_p}c_p$ for its expansion 
in an alternate basis and conclude that the monomial factor of the coefficients 
$m_{b,c_j}\in R$ must satisfy $\mbf{b}=\mdeg(m_{b,c_j})+\mbf{c}_j$ for all $j$. 

Since $f_i$ is an automorphism of a free module then it must be 
the product of elementary matrices which correspond to one of the 
following column operations: 
\begin{enumerate}
\item interchange two columns in $d'_i$
\item multiply a column in $d'_i$ by an element of $R$
\item add a ring element multiple of a column in $d'_i$ to another column.
\end{enumerate}       
We consider each case individually.

Since $M$ is rigid, the basis elements in the free module $F_i$ 
have unique multidegrees due to condition (R1).  Furthermore, 
these multidegrees correspond to monomials which are 
pairwise incomparable in $\LCM(M)$ due to condition (R2).  
We therefore know that the free module 
$F_i$ contains only one shifted copy of the ring for each multidegree.

Suppose that $F_i$ may be realized using either the choice of multigraded 
basis $B_i$ or the choice of multigraded basis $C_i$. 
Without loss of generality, we may assume that the bases $B_i$ and 
$C_i$ have a fixed ordering in which $b_t$ and $c_s$ 
have the same multidegree if and only if $t=s$.  This rules out 
the possibility of maps of type 1 occurring in the structure of $f_i$.  

For maps of type 2, the equality of the multidegrees of the target 
and source basis elements forces each of the nonzero entries of the matrix 
of $f_i$ to be a unit in $\Bbbk$ since the degree preserving nature 
of the map must be upheld.  

Finally, considering maps of type 3, we see that the fact that 
$\mbf{b}_k$ and $f_i(\mbf{b}_k)$ have the same multidegree for each $k$ 
implies that if $f_i$ adds a multiple of a column $l$ to the column 
$k$ then in order for multidegrees to be preserved, it must be that 
$\mbf{x}^{\mbf{b}_l}$ divides $\mbf{x}^{\mbf{b}_k}$.  This however, 
implies that either $\mbf{x}^{\mbf{b}_l}$ and $\mbf{x}^{\mbf{b}_k}$ 
are comparable in $\LCM(M)$ or that they are equal, 
contradicting condition (R2) or (R1) respectively.  

Together, these facts imply that the matrix of $f_i$ must be 
invertible, diagonal and degree preserving so that it has only 
units along its diagonal.  Thus, the only automorphisms of 
$\mathcal{F}$ are trivial and specifically, the differential maps in 
$\mathcal{F}$ are unique up to scaling.  

Conversely, suppose that a monomial ideal $M$ has a free resolution 
which admits only trivial automorphisms.  Since the individual maps 
in such an automorphism cannot be constructed from a product of elementary 
matrices which interchange two columns, the multidegrees of 
the free modules must be unique.  Since these maps also cannot be 
constructed from a product of elementary matrices which multiply a 
column in $d'_i$ by an element of $R$ or which add a ring element 
multiple of a column in $d'_i$ to another column, the multidegrees of 
the free modules are not comparable to one another in $\LCM(M)$.  Thus, 
conditions (R1) and (R2) are satisfied and $M$ is a rigid ideal.  
\end{proof}

%%%%%%%%%%%%%%%%%%%%%%%%%%%%%%%%%%%%%%%%%%%%%%%
%%%%%%%%%%%%%%%%%%%%%%%%%%%%%%%%%%%%%%%%%%%%%%%

\section{Minimal resolutions as poset resolutions}

The benefit of looking at rigid monomial ideals instead of the entire 
class of monomial ideals is that we have some hope of writing 
down a closed form description of the minimal free resolution.  
We do not give a complete description here, but provide descriptions 
of resolutions for a subclass of rigid monomial ideals and 
describe the remaining types to consider.

We begin by introducing a notion that describes the multidegrees 
contributing to the minimal resolution in the context of the poset 
relations of the lcm-lattice.  We call $M$ {\it concentrated} 
if every multidegree from $\LCM(M)$ which 
does not contribute to the minimal free resolution appears higher 
than all the contributing multidegrees in $\LCM(M)$.  
In other words, every multidegree which is smaller than a 
contributing multidegree must itself contribute.  
Formally, we have the following.  

\begin{definition}
A monomial ideal is said to be \emph{concentrated} if 
it has the property that for every $\mbf{x^a}\in \LCM(M)$ such that 
$\beta_{j,\mbf{a}}(R/M) = 0$ for all $j$ then $\mbf{x^a} > \mbf{x^b}\in \LCM(M)$ 
for every $\mbf{x^b}\in \LCM(M)$ for which $\beta_{i,\mbf{b}}(R/M) \ne 0$ 
for some $i$. 
A monomial ideal which is not concentrated is said 
to be \emph{dispersed}.  
\end{definition}

The class of concentrated monomial ideals is a generalization 
of the class of lattice-linear monomial ideals, whose minimal free resolution 
was constructed in \cite{Clark}.  

\begin{definition}\label{LLDef}
A monomial ideal $M$ is \emph{lattice-linear} if multigraded
bases $B_k$ of the free modules $F_k$ appearing in the minimal 
free resolution of $R/M$ can be fixed for all $k$ 
so that for any $i\ge 1$ and any $b\in{B_i}$ the differential 
$$d_i(b)=\sum_{b'\in{B_{i-1}}}m_{b,b'}\cdot{b'}$$ 
has the property that if the coefficient $m_{b,b'}\ne 0$ 
then $\mbf{x}^{\mbf{b'}}$ is covered by $\mbf{x}^{\mbf{b}}$ in the 
lcm-lattice $\LCM(M)$.
\end{definition}

The property of being concentrated is indeed a generalization 
of lattice-linearity, for if an ideal is lattice-linear, the coverings 
in the lcm-lattice are mirrored by the action 
of the differential on the corresponding basis elements.  
Were any noncontributing element dispersed between 
contributing elements within $\LCM(M)$, the assumption of 
lattice-linearity would immediately be contradicted.  Thus, 
every ideal which is lattice-linear is a concentrated ideal.  

\begin{remark}\label{ConcNotLL}
For an arbitrary monomial ideal, the property of being 
concentrated is not enough to guarantee lattice-linearity.  

Consider the example of the monomial ideal whose 
lcm-lattice is the augmented face poset of the simplicial 
complex on six vertices which consists of three triangles 
attached pairwise at three vertices.  Note here that to 
create the {\it augmented face poset} a maximal element 
is introduced to the existing face poset 
(which is a meet-semilattice) in order to create a finite 
atomic lattice.  One coordinatization of this 
lcm-lattice has as a minimal generating set consisting of the monomial 
vertex labels in Figure \ref{ConcentratedNotRigid}.  
The monomial ideal $M$ is concentrated since it has a ranked lcm-lattice, 
but is clearly not lattice-linear since the free module whose 
multidegree matches $a^2bcu^2vwx^2yz$, the top element of $\LCM(M)$, 
appears in homological degree three.  The free modules whose multidegrees 
are covered by $a^2bcu^2vwx^2yz$ in the lcm-lattice also appear 
in homological degree three.  
\end{remark}

%$$\langle cu^2vwx^2yz,bwx^2yz,a^2bcvx^2yz,au^2vwz,a^2bcuy,a^2bcu^2vwx\rangle.$$  
\begin{figure}
\begin{center}
\begin{tikzpicture}
\draw[fill=lightgray] (0,0) -- (2,0) -- ++ (120:2) -- (0,0) -- cycle; % ++ means relative to last coordinate
\draw[fill=lightgray] (2,0) -- (4,0) -- ++ (120:2) -- (2,0) -- cycle;  % ++ means relative to last coordinate
\draw[fill=lightgray] (1,1.732) -- (3,1.732) -- ++ (120:2) -- (1,1.732) -- cycle;  % ++ means relative to last coordinate
\draw[fill] (0,0) circle (1pt) node[left]{$cu^2vwx^2yz$};
\draw[fill] (2,0) circle (1pt) node[below]{$bwx^2yz$};
\draw[fill] (4,0) circle (1pt) node[right]{$a^2bcvx^2yz$};
\draw[fill] (1,1.732) circle (1pt) node[left]{$au^2vwz$};
\draw[fill] (3,1.732) circle (1pt) node[right]{$a^2bcuy$};
\draw[fill] (2,3.464) circle (1pt) node[above]{$a^2bcu^2vwx$};
\end{tikzpicture}
\end{center}
\caption{A concentrated ideal which is not rigid}\label{ConcentratedNotRigid}
\end{figure}
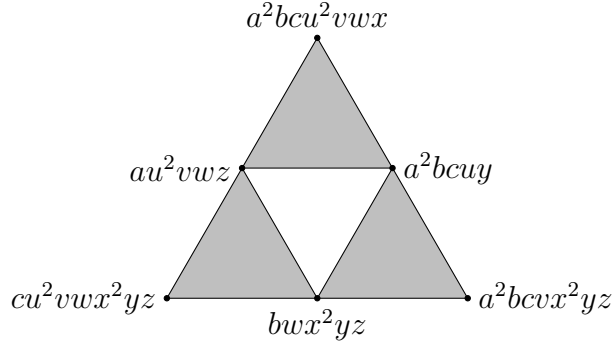

For the class of rigid monomial ideals however, the properties 
of being concentrated and being lattice-linear coincide.  

\begin{theorem}\label{concentratedLatticeLinear}
A rigid monomial ideal is concentrated if and only if it is lattice-linear.
\end{theorem}

\begin{proof} 
We need only show that a rigid ideal which is concentrated must be lattice-linear 
and proceed by proving the contrapositive of this implication.  

Suppose that $M$ is a rigid monomial ideal which is not lattice-linear.  
By the definition of rigidity, the automorphisms of the minimal free 
resolution of $M$ are trivial and $M$ has a unique minimal free resolution 
$(\mathcal{F},d)$.  For each $\ell$, write $B_\ell$ for the unique choice of basis of the 
free module $F_\ell$ appearing in the resolution.  Our supposition 
that $M$ is not lattice-linear implies that for some $i$, there exists a basis 
element $c$ of multidegree $\mbf{c}$ in $B_i$, such that there exists some 
multidegree $\mbf{a}$ with the property that the coefficient $m_{a,c}$ 
in the expansion of $d_i(c)$ is nonzero but that $\mbf{x^a}$ is not 
covered by $\mbf{x^c}$ in $\LCM(M)$.  In other words, there exists 
$\mbf{x^b}\in \LCM(M)$ such that $\mbf{x^a}\lneq \mbf{x^b} \lneq \mbf{x^c}$.  

Since $M$ is assumed to be rigid, we have $\beta_{j,\mbf{b}}(R/M) \ne 0$ for some $j$.  
Since $\mbf{x^a}\lneq \mbf{x^b} \lneq \mbf{x^c}$ there exist multigraded strands 
within the free resolution which have the following containment structure;
 
$$\mathcal{F}_{\le \mbf{a}}\subset\mathcal{F}_{\le \mbf{b}}\subset\mathcal{F}_{\le \mbf{c}}.$$

This structure is contradicted if $j>i$ since 
the degree $\mbf{c}$ strand terminates in homological degree $i$ 
and the degree $\mbf{b}$ strand terminates in homological degree $j$.  
If $j=i$ (or respectively $j=i-1$), then we have identified two multidegrees which 
contribute in homological degree $i$ (or respectively $i-1$) that are comparable to 
one another in $\LCM(M)$, a contradiction to the second property in 
the definition of rigidity.  Lastly, if $j<i-1$ the strand inclusion 
above is again contradicted since the degree $\mbf{b}$ strand must terminate 
at a higher homological degree than the degree $\mbf{a}$ strand.  

Such an integer $j$ therefore does not exist and $\beta_{\ell,\mbf{b}}(R/M) = 0$ 
for every $\ell$.  Hence, if a rigid ideal is not lattice-linear then it is not concentrated.  
\end{proof}

%%%%%%%%%%%%%%%%%%%%%%%%%%%%%%%%%%%%%%%
%%%%%%%%%%%%%%%%%%%%%%%%%%%%%%%%%%%%%%%
\section{Stability of resolutions for subsets of  $\mathcal{L}(n)$}\label{filters}

Let $\mathcal{L}(n)$ be the set of all finite atomic lattices with $n$ ordered atoms.  
Phan in \cite{phan} defines a partial order on $\mathcal{L}(n)$ as $P\geq Q$ if there 
exists a join preserving map $f:P \rightarrow Q$ which is a bijection on atoms.  
Theorem 4.2 in \cite{phan} shows that under this partial order, $\mathcal{L}(n)$ is 
itself a finite atomic lattice.  Moreover, Theorem 3.3 in \cite{GPW} indicates that 
total Betti numbers weakly increase as one moves up chains in $\mathcal{L}(n)$.  
Additionally, Theorem 5.1 in \cite{phan} shows that every finite atomic lattice is the 
lcm-lattice of some monommial ideal.  Thus, $\mathcal{L}(n)$ can be thought of as 
the lattice of all monomial ideals with $n$ ordered generators up to equivalence of lcm-lattices.  

Using the formulas for multigraded Betti numbers which utilize order complexes of 
intervals in the lcm-lattice (see Theorem 2.1 in \cite{GPW}), we can interchangeably 
refer to the Betti numbers of a finite atomic lattice and the Betti numbers of a 
monomial ideal.  Specifically, for a monomial $\mbf{x^b} \in \LCM(M)$ the formula 
for computing multigraded Betti numbers in homological degree $i$ is:
$$\beta_{i,\mbf{b}}(R/M) = \dim \tilde{H}_{i-2}(\Delta (\hat{0},\mbf{x^b}); \Bbbk),$$ 
where $\Delta (\hat{0},\mbf{x^b})$ is the order complex of the open interval from $\hat{0}$ to $\mbf{x^b}$.  
Since rigid monomial ideals are defined by the behavior of their multigraded Betti 
numbers, we call a finite atomic lattice rigid (or not) using the same definition.  
In this section we will often refer to the Betti numbers of a finite atomic lattice $P$ 
rather than the Betti numbers of a specific monomial ideal.  As such, we use 
as $\beta_{i,p}$ where $p \in P$ to denote these Betti numbers.  

Let $\beta_i = \sum \beta_{i,\mbf{b}}$ be the total Betti numbers of the ideal $M$ and write 
$\beta = (\beta_0, \beta_1, \dots, \beta_t)$ for the vector of total Betti numbers.  
It is reasonable then to fix subposets of $\mathcal{L}(n)$ which consist of all the 
finite atomic lattices with the same total Betti numbers.  We refer to these 
subposets as {\it Betti stratum} and denote them $\mathcal{L}(n)_{\beta}$.  Given a rigid 
monomial ideal $M$ whose total Betti numbers are $\beta$ we now examine the relationship 
between $M$ and ideals whose lcm-lattices are in $\mathcal{L}(n)_{\beta}$.

\begin{proposition}\label{rigidCover}
Let $Q,P \in \mathcal{L}(n)_{\beta}$ for some $\beta$ such that $Q$ covers $P$ 
(i.e. $Q>P$ and there is no lattice $T$ such that $Q>T>P$).  If $P$ is rigid then $Q$ is rigid.   
\end{proposition}

\begin{proof}

From proposition 5.1 in \cite{mapes} we know that if $Q$ 
covers $P$ then as a set $Q = P \cup \{q\}$ for some $q$.  
Moreover, since $P$ is a finite atomic lattice, we know that $q \in Q$ 
must be meet-irreducible (it is not the meet of any two elements in $P$).  
So there is a unique element $p' \in Q$ which covers $q$.  To check that 
$Q$ is rigid we need to check conditions (R1) and (R2).  Note that for all 
$p<q$ in $Q$, the interval $(\hat{0}, p)$ in $Q$ is identical to the interval 
$(\hat{0}, p)$ in $P$ so the associated Betti number is unchanged.  
Similarly, the Betti number is unchanged for all elements $p \in Q$ 
which are not comparable to $q$.  

Finally for any $p \gneq p'$ in $Q$ (or $P$ for that matter) we will 
show that the Betti number is also unchanged.  To see this, consider 
the join preserving bijection on atoms from $Q \to P$.  The fiber over a 
point $p\in P$ is $\{p\}$ for $p \neq p'$ or $\{p,q\}$ in the case when $p = p'$.  
Restrict this map to those intervals $(\hat{0}, p)$ in $Q$ where $p \gneq p'$.  
The fibers of this map are contractible and by Quillen's Fiber Theorem 
(see Theorem 10.5 in \cite{bjorner}) the order complexes of these intervals 
are homotopy equivalent.  

It is therefore only necessary to check the conditions of rigidity 
for multidegrees corresponding to $p'$ and $q$.  Since the total 
Betti numbers of $P$ and $Q$ are the same and for all other multidegrees 
$p \in Q$ we have seen that $\beta_{i,p}$ is the same as in $P$, 
if $\beta_{i, p'} = 0$ in $P$ for all $i$ then the same is true 
in $Q$ and $\beta_{i,q}=0$ in $Q$.  Otherwise either $\beta_{i,q}$ or 
$\beta_{i,p'}$ is 1 in $Q$ for some $i$.  Thus condition (R1) is satisfied.

To see that condition (R2) holds, consider the following.  
If $\beta_{i,p} = 1$ in $Q$ then since $P$ is rigid, condition (R2) is 
satisfied for $Q$ as well.  Alternatively, if $\beta_{i,q} = 1$ in $Q$ 
and condition (R2) is not satisfied then there is some $p \in Q$ such 
that $\beta_{i,p} = 1$ and $q$ and $p$ are comparable in $Q$.  By the 
above argument however, we know that $\beta_{i,p} = 1$ in $P$ and 
so if $p$ and $q$ are comparable then $p$ must also be comparable to $p'$ 
which contradicts the fact that $P$ is rigid.  As such, (R2) must be 
satisfied for $Q$.
\end{proof}

An easy corollary of this is the following.

\begin{corollary}\label{rigidFilter}
Let $P$ and $Q$ be in the same Betti stratum $\mathcal{L}(n)_{\beta}$ with $Q>P$.  If $P$ is rigid then $Q$ is rigid.
\end{corollary}

The following extends the result of Theorem \ref{concentratedLatticeLinear} 
and allows the construction of the minimal free resolution of all dispersed rigid 
monomial ideals within a Betti stratum which are greater than a concentrated rigid 
monomial ideal appearing in the same stratum. 

\begin{theorem}\label{resInRigidFilter}
Let $P, Q \in \mathcal{L}(n)_{\beta}$ for some $\beta$.  If $P$ is rigid and $Q > P$ then the minimal resolution of $P$ is isomorphic to the minimal resolution of $Q$.
\end{theorem}

\begin{proof}
By Corollary \ref{rigidFilter} we know that $Q$ is rigid.  Using the fact 
that $Q>P$ means that there exists a join preserving map $f:Q \rightarrow P$ 
which is a bijection on atoms.  We can therefore apply Theorem 3.3 in \cite{GPW} 
which says: if $\mathcal{F}$ is a minimal resolution of $Q$ then $f(\mathcal{F})$ 
is a resolution of $P$.  Since the total Betti numbers are the same for both $P$ 
and $Q$ it follows that $f(\mathcal{F})$ is a minimal resolution of $P$. 

Using rigidity we see that $P$ has a unique minimal resolution up to scaling, 
thus $\mathcal{F}$ must be isomorphic to the minimal resolution of $P$.  
Moreover since $Q$ is itself rigid, the minimal resolution of $Q$ is isomorphic to that of $P$. 
\end{proof}

\begin{remark}
In the situation when one knows the minimal free resolution of a monomial ideal whose lcm-lattice is 
$Q$, the ability to transfer this resolution information to the minimal resolution of 
$P$ appearing lower in the stratum $\mathcal{L}(n)_{\beta}$ is an artifact of Theorem 3.3 in \cite{GPW}.  
The new information of Theorem \ref{resInRigidFilter} is that if the minimal resolution of the rigid 
ideal $P$ is known then it can be used to construct the minimal resolution of $Q$, the ideal appearing higher 
in the stratum $\mathcal{L}(n)_{\beta}$.  This is not true if $P$ is not rigid.  
\end{remark}

In the following example we see how to combine Theorem \ref{concentratedLatticeLinear} and Theorem \ref{resInRigidFilter} to construct minimal resolutions for a larger set of rigid monomial ideals than just the concentrated ones.

\begin{example}
Figure \ref{concentratedRigidFilter} shows an interval of the Betti stratum $\mathcal{L}(4)_{(1,4,4,1)}$.  
The finite atomic lattice at the bottom represents the lcm-lattice of a concentrated 
rigid monomial ideal, and therefore its minimal resolution can be constructed by using 
the poset resolution construction on the lcm-lattice.  

\begin{figure}
\center
\includegraphics[scale=0.25]{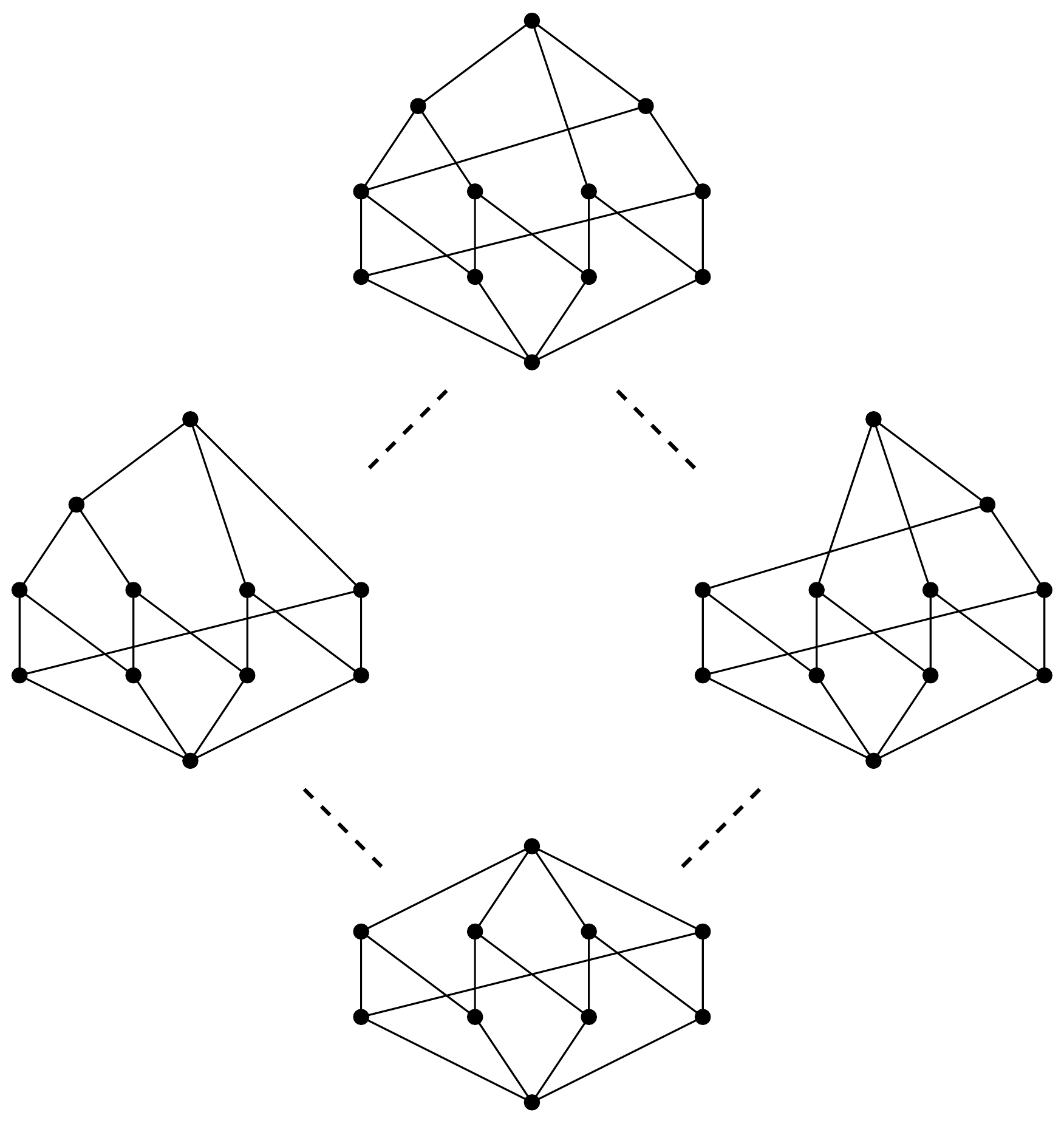}
\caption{Rigid ideals whose minimal resolution can be constructed as a poset resolution}\label{concentratedRigidFilter}
\end{figure}

The finite atomic lattices above this lattice represent lcm-lattices of dispersed rigid monomial 
ideals, since the maximal element in all of the lattices corresponds to a second syzygy.  
Due to Theorem \ref{resInRigidFilter} each of these ideals admit a minimal poset resolution 
whose maps mirror those of the concentrated rigid ideal.  
\end{example}

The following proposition and corollary show how we can use 
rigidity to understand some cellular resolutions.

\begin{proposition}\label{faceRigid}
Let $P_X$ be the augmented face poset of a regular CW complex $X$.  
If $P_X$ is a finite atomic lattice and $X$ is acyclic, then $P_X$ is rigid.
\end{proposition}

\begin{proof}
By Proposition 6.5 in \cite{mapes} we know that a minimal resolution of $P_X$ 
is supported on $X$.  Clearly each cell $\alpha$ contributes a different 
Betti degree and no two cells of the same dimension are contained in each other.  
Thus the non-comparability condition is satisfied.
\end{proof}

In fact, using rigidity and the previous proposition we can improve 
Proposition 6.5 in \cite{mapes} with this easy corollary to Theorem \ref{resInRigidFilter}. 

\begin{corollary}
Let $P_X$ be the augmented face poset of an acyclic regular CW complex $X$.  Let $\beta$ be the vector of total Betti numbers for $P_X$.  If $P_X$ is a finite atomic lattice and $Q > P_X\in \mathcal{L}(n)_{\beta}$ then $Q$ has a minimal resolution supported on $X$.
\end{corollary}

\begin{proof}
By Proposition \ref{faceRigid} we see that $P_X$ is rigid.  Thus, Theorem \ref{resInRigidFilter} guarantees that the minimal resolution of $P_X$ will give a minimal resolution of $Q$.   Since the minimal resolution of $P_X$ is supported on $X$ by Proposition 6.5 in \cite{mapes} then the minimal resolution of $Q$ is also supported on $X$. 
\end{proof}

\section{Directions for future work}

The following example shows the limitations of resolving dispersed rigid monomial ideals 
by obtaining resolution information from concentrated rigid ideals appearing lower in the same Betti stratum.  Indeed, for a given ideal, such comparable ideals need not exist.  We are unaware of any conditions which guarantee the existence of such a comparable concentrated rigid ideal.  

\begin{example}\label{2squares}
The rigid monomial ideal $M=\langle bd,cd^2,ac,c^2d,ab\rangle$ is minimally 
resolved on the 3-dimensional regular CW complex $X$ pictured in Figure \ref{RigidDispersedCWposet}, 
where the multidegree of each vertex matches that of the monomial appearing in the given ordered 
list of generators.  Note that $X$ has $(1,5,7,4,1)$ as its face vector.  

The face poset of this cell complex, $P_X$, is not a meet semi-lattice and therefore 
is not contained in any Betti stratum of $\mathcal{L}(5)$.  Furthermore, 
the lattice $\LCM(M)$ is $P_X \cup \{p= a_1 \vee a_3 \vee a_5\}$ where the $a_i$ 
are the atoms in $P_X$ corresponding to vertex $i$.  The element $p \in \LCM(M)$ 
does not correspond to a Betti degree and therefore $M$ is a dispersed rigid monomial 
ideal.  Since removing the element $p$ from $\LCM(M)$ produces a poset which is 
not a lattice, there is no finite atomic lattice less than $\LCM(M)$ which appears in the 
same stratum and is a rigid concentrated monomial ideal.  We therefore cannot apply 
the combination of Theorem \ref{resInRigidFilter} and Theorem \ref{concentratedLatticeLinear} 
to construct a resolution of $M$.  
\end{example}

\begin{figure}
\center
\begin{tikzpicture}[scale=1.5]
\draw[fill=lightgray](0,0,0) -- (1,0,0) -- (1,0,1) -- (0,0,1) -- cycle; 
\draw[fill=lightgray] (0,0,0) -- (0,1,0) -- (0,0,1) -- cycle; 
\draw[fill=lightgray] (0,0,0) -- (0,1,0) -- (1,0,0) -- cycle; 
\draw[fill] (0,0,0) circle (1pt) node[below]{$4$};
\draw[fill] (1,0,0) circle (1pt) node[right]{$3$};
\draw[fill] (0,1,0) circle (1pt) node[left]{$2$};
\draw[fill] (0,0,1) circle (1pt) node[left]{$1$};
\draw[fill] (1,0,1) circle (1pt) node[below]{$5$};
\draw (-1,1,0) node[above]{$X$};
\end{tikzpicture}
\caption{The regular CW complex of Example \ref{2squares}}\label{RigidDispersedCWposet}
\end{figure}
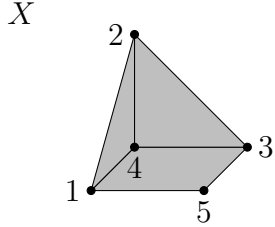

In actuality one may consider the following reinterpretation of Theorems \ref{concentratedLatticeLinear} and \ref{resInRigidFilter}.  Define $\beta(\LCM(M))$ to be the subposet of $\LCM(M)$ consisting of $p \in \LCM(M)$ such that $\beta_{i,b} \neq 0$ for some $i$.  Call this the Betti subposet of $\LCM(M)$.  In the situation where $M$ is concentrated rigid $\beta(\LCM(M))$ is the union of intervals $[\hat{0},p]$ in $\LCM(M)$ where $\beta_{i,p}$ is a Betti number.  Thus one can compute the ranks of the free modules by using either $\LCM(M)$ or $\beta(\LCM(M))$ as the pertinent information has not changed.  Notice too that if $M'$ is a rigid monomial ideal such that $\LCM(M') > \LCM(M)$ and $M'$ and $M$ are in the same Betti stratum then $\beta(\LCM(M)) = \beta(\LCM(M'))$.  So we can say that the minimal free resolution for both $M$ and $M'$ can be constructed as a poset resolution on $\beta(\LCM(M))$.  

We believe that this approach should be applicable to all rigid monomial ideals.  
In Example \ref{2squares} we can realize the minimal resolution as a poset resolution 
on $\beta(\LCM(M)) = P_X$.  Indeed, removing the element $p =  a_1 \vee a_3 \vee a_5$ 
does not change the homology of the order complex for intervals $(\hat{0}, q)$ where $q > p$ in $P_X$.  

\begin{question}
Is the minimal free resolution of any rigid monomial ideal $M$ a poset resolution on $\beta(\LCM(M)$?
\end{question}

\end{document}